\documentclass[a4paper]{amsart}

\usepackage{mathrsfs,amssymb}

\usepackage{enumerate}
\usepackage{amsmath}
\usepackage{amssymb}
\usepackage{latexsym}
\usepackage{mathrsfs}
\usepackage{graphicx}

\newtheorem{theorem}{Theorem}[section]

\newtheorem{lemma}[theorem]{Lemma}
\newtheorem{proposition}[theorem]{Proposition}
\newtheorem{corollary}[theorem]{Corollary}

\theoremstyle{definition}
\newtheorem{definition}[theorem]{Definition}
\newtheorem{remark}[theorem]{Remark}
\newtheorem{example}[theorem]{Example}
\newtheorem{question}[theorem]{Question}

\newcommand{\outball}{\overrightarrow{\mathcal{B}}}

\newtheorem*{lem_dual}{Lemma \ref{lemma_bookkeeping}'}

\newcommand{\lb}{\langle}
\newcommand{\rb}{\rangle}

\usepackage{amssymb}

\newcommand{\A}{\mathcal{A}}

\newcommand{\gl}{\mathscr{L}}
\newcommand{\gr}{\mathscr{R}}

\newcommand{\gh}{\mathscr{H}}

\usepackage{verbatim}

\newcommand{\defterm}[1]{\textsl{#1}}

\newcommand{\nset}{\mathbb{N}}
\newcommand{\zset}{\mathbb{Z}}

\newcommand{\emptyword}{\varepsilon}

\newcommand{\rel}[1]{\mathscr{#1}}
\newcommand{\elt}[1]{\overline{#1}}

\newcommand{\dotcup}{\ensuremath{\mathaccent\cdot\cup}}

\begin{document}
\title[Green index in semigroups]{Green index in semigroups: generators, presentations and automatic structures}


\maketitle

\begin{center}

    ALAN J. CAIN\footnote{Supported by the
project PTDC/MAT/69514/2006 ‘Semigroups and Languages’, funded by FCT
and PIDDAC.}

    \medskip

    Centro de Matem\'{a}tica da Universidade do Porto,\\ Rua do Campo Alegre 687, \\
    4169-007 Porto, Portugal.

    \medskip

    \texttt{a.j.cain@gmail.com}

    \bigskip

    \bigskip

    ROBERT GRAY\footnote{Supported by an EPSRC Postdoctoral Fellowship.}, \quad NIK RUSKUC

    \medskip

    School of Mathematics and Statistics,\ University of St Andrews, \\
    St Andrews KY16 9SS, Scotland.

    \medskip

    \texttt{robertg@mcs.st-and.ac.uk}, \quad \texttt{nik@mcs.st-and.ac.uk} \\
\end{center}

\begin{abstract} 
Let $S$ be a semigroup and let $T$ be a subsemigroup of $S$. 
Then $T$ acts on $S$ by left- and by right multiplication.
This gives rise to a partition of the complement $S \setminus T$ and to each equivalence class of this partition we naturally associate a relative Sch\"{u}tzenberger group.  
We show how generating sets for $S$ may be used to obtain generating sets for $T$ and the Sch\"{u}tzenberger groups, and vice versa. 
We also give a method for constructing a presentation for $S$ from given presentations of $T$ and the  Sch\"{u}tzenberger groups. 
These results are then used to show that several important properties are preserved when passing to finite Green index subsemigroups or extensions, including: finite generation, solubility of the word problem, growth type, automaticity, finite presentability (for extensions) and 
finite Malcev presentability (in the case of group-embeddable semigroups). These results provide common generalisations of several classical results from group theory and Rees index results from semigroup theory. 
\end{abstract}

\section{Introduction}
\label{sec_intro}
The notion of the index of a subgroup is a fundamental concept in group theory. It may be viewed as providing a way of measuring the size of a subgroup relative to its containing group. From this point of view, a subgroup of finite index may be thought of as differing from its parent by only a finite amount. This intuitive idea gains more significance when one considers the long list of properties that are known to be preserved when passing to finite index subgroups or extensions, which include: finite generation and presentability (and more generally property ${\rm F}_n$ for every $(n\geq 1)$), solubility of the word problem, automaticity, the homological finiteness property ${\rm FP}_n$, 
residual finiteness, periodicity, and local-finiteness (see \cite{Brown1982,delaHarpe2000,epstein_wordproc,lyndon_cgt,Magnus1} for details of these classical results). 
On the other hand, it is still an open question as to whether the property of being presented by a finite complete rewriting system is inherited by subgroups of finite index; see \cite{Pride2000}. 
Important problems about finite index subgroups and extensions continue to receive a great deal of attention; see for example \cite{Behrstock1,Burillo1,Haglund1,Katsuya1,Nikolaev1,Nikolov2003,Nikolov1,Nikolov2007(1),Nikolov2007(2)}. 

In semigroup theory various notions of index have arisen, in several different contexts. For example, the index of a subgroup of a semigroup was considered by Bergman in \cite{Bergman}, while a notion of index for cancellative semigroups arose in work of Grigorchuk \cite{Grigorchuk1} on growth of semigroups. While these are direct generalisations of group index, they are limited since they do not apply to semigroups in general. 

For subsemigroups of semigroups in general, until recently, the most widely studied notion of index 
has been the so-called Rees index. The \emph{Rees index} of a subsemigroup is defined simply as the cardinality of its complement. It was originally introduced by Jura in \cite{Jura1}, and since then, in 
analogy with group index, many finiteness conditions have now been shown to be inherited when passing to finite Rees index substructures or extensions (see \cite{Campbell1996}, \cite{Ruskuc1} and \cite{hoffmann_autofinrees}). 

However, Rees index is not a generalisation of group index. 
In fact, it is obvious that an infinite group cannot have any proper subgroups of finite Rees index. So although Rees index results have the same look and feel as group index results, this is as far as the connection goes, 
and in particular they cannot be applied to recover the corresponding group-theoretic results on which they are modelled.   
This problem was the original motivation for the work in \cite{gray_green1} where a new notion of index was introduced, called \emph{Green index}.
The Green index of a subsemigroup $T$ of a semigroup $S$ is given by counting strong orbits (called $T$-relative $\gh$-classes) in the complement  $S \setminus T$ under the natural actions of $T$ on $S$ via right and left multiplication (see Section~2 for more details). In particular, when $S \setminus T$ is finite $T$ will have finite Green index in $S$, while if $S$ is a group and $T$ a subgroup then $T$ will have finite Green index in $S$ if and only if it has finite group index in $S$. In \cite{gray_green1} it was shown that several important finiteness conditions are preserved when taking finite Green index subsemigroups or extensions. Thus, Green index is both general enough to simultaneously subsume Rees index and group index, but also strong enough that a semigroup will share many interesting properties with its subsemigroups of finite Green index. In this paper we continue the investigation of Green index, and in particular extend the list of finiteness conditions that are known to be preserved under taking finite Green index subsemigroups and extensions.

Extending the classical ideas of Sch\"{u}tzenberger \cite{Schutzenberger1957,Schutzenberger1958}, 
with each $T$-relative $\gh$-class $H$ we associate a group $\Gamma(H)$, called its ($T$-\emph{relative}) \emph{Sch\"{u}tzenberger group}, obtained by taking the setwise stabiliser of the action of $T$ on $H$ by right multiplication and making it faithful (see Section~2 for full details). Our results then relate properties of $S$, $T$ and the family of (relative) Sch\"{u}tzenberger groups $\Gamma(H)$.

The article is laid out as follows. After the preliminaries in Section \ref{sec_prelims}, in Section \ref{sec_rewriting} we prove a fundamental lemma (the Rewriting Lemma) which underpins many of the results appearing later in the paper.  
This rewriting technique   is utilised in Section \ref{secsubsemigroups} to
obtain a generating set for $T$ from a generating set for $S$.
In Section \ref{secschgps} we obtain generating sets for the relative
Sch\"{u}tzenberger groups from a generating set for $S$.
In the case of finite Green index, finite generation is preserved in both these situations. 
In Section \ref{sec_presentations} we give a presentation for $S$ in terms of given presentations for $T$ and each of the Sch\"{u}tzenberger groups. Again, when the Green index is finite, finite presentability is preserved. 
Whether finite presentability is preserved in the other direction, i.e. from $S$ to $T$ and the Sch\"{u}tzenberger groups, remains an open problem, but in Section \ref{secmalcev}
we show that this is the case for finite Malcev (group-embeddable) presentations (in the sense of \cite{c_survey}).
In the remaining sections we consider a range of other properties related to generators in one way or another:
the word problem (Section \ref{secwp}), type of growth (Section \ref{secgrowth}),
and automaticity (Section \ref{secautomatic}) in the sense of \cite{campbell_autsg,hoffmann_relatives}.
These results provide common generalisations of the corresponding classical results from group theory, and Rees index results from semigroup theory. 

\section{Preliminaries}
\label{sec_prelims}

Let $S$ be a semigroup and let $T$ be a subsemigroup of $S$. 
We use $S^1$ to denote the semigroup $S$ with an identity element $1 \not\in S$ adjoined to it. This notation will be extended to subsets of $S$, i.e. $X^1 = X \cup \{ 1 \}$. For $u,v \in S$ define:
\[
u \gr^T v  \ \Leftrightarrow  \ uT^1 = vT^1, \quad u \gl^T v \  \Leftrightarrow \  T^1u = T^1v,
\]
and $\gh^T = \gr^T \cap \gl^T$. Each of these relations is an equivalence relation on $S$; their equivalence classes are called the
($T$-)\emph{relative}  $\gr$-, $\gl$-, and $\gh$-classes, respectively.
Furthermore, these relations respect $T$, in the sense that each
$\gr^T$-, $\gl^T$-, and $\gh^T$-class lies either wholly
in $T$ or wholly in $S \setminus T$. Following \cite{gray_green1} we define the \defterm{Green index} of $T$
in $S$ to be one more than the number of relative $\gh$-classes in $S \setminus T$. Relative Green's relations were introduced by Wallace in \cite{Wallace} generalising the the fundamental work of Green \cite{Green1951}. For more on the classical theory of Green's relations, and other basic concepts from semigroup theory, we refer the reader to \cite{howie_fundamentals}.

Throughout this paper $S$ will be a semigroup, $T$
will be a subsemigroup of $S$, and Green's relations in $S$ will
always be taken relative to $T$, unless otherwise stated. In other words, we shall write $x \gr y$ to mean that $xT^1 = yT^1$ rather than $xS^1 = yS^1$. On the few occasions that we need to refer to Green's $\gr$ relation in $S$ we will write $\gr^S$. The same goes for the relations $\gl$ and $\gh$.

The following result summarises some basic facts about relative Green's relations (see \cite{Wallace,gray_green1} for details).

\begin{proposition}\label{basicproperties}
Let $S$ be a semigroup and let $T$ be a subsemigroup of $S$.
\begin{enumerate}[(i)]
\item
\label{prop_LandRcong}
The relative Green's relation $\gr$ is a left congruence on $S$, and $\gl$ is a right congruence.
\item
\label{prop_{Rel_Greens_Lemma}}
Let $u,v \in S$  with $u \gr v$, and let $p, q \in T$ such that $up = v$ and $vq = u$. Then the mapping $\rho_p$ given by $x \mapsto xp$ is an $\gr$-class preserving bijection from $L_u$ to $L_v$, the mapping $\rho_q$ given by $x \mapsto xq$ is an $\gr$-class preserving bijection from $L_v$ to $L_u$, and $\rho_p$ and $\rho_q$ are mutually inverse.
\end{enumerate}
\end{proposition}

\begin{sloppypar}
With each relative $\gh$-class we may associate a group, which we call the \emph{Sch\"{u}tzenberger group} of the $\gh$-class.
This is done by extending, in the obvious way, the classical definition (introduced by Sch\"{u}tzenberger in \cite{Schutzenberger1957,Schutzenberger1958}) to the relative case.
For each $T$-relative $\gh$-class $H$ let
$\mathrm{Stab}(H) = \{ t \in T^1 : Ht = H \}$
(the \emph{stabilizer} of $H$ in $T$), and
define an equivalence $\gamma=\gamma(H)$ on $\mathrm{Stab}(H)$ by $(x,y) \in \gamma$ if and only if $hx = hy$
for all $h\in H$. Then $\gamma$ is a congruence on $\mathrm{Stab}(H)$ and $\mathrm{Stab}(H) / \gamma$ is a group.
The group $\Gamma(H)=\mathrm{Stab}(H) / \gamma$ is called the relative \emph{Sch\"{u}tzenberger group} of $H$.
The following basic observations about relative Sch\"{u}tzenberger groups will be needed (see \cite{Wallace,gray_green1} for details).
\end{sloppypar}

\begin{proposition}
\label{schutzstabproperties}
Let $S$ be a semigroup, let $T$ be a subsemigroup of $S$, let $H$ be a relative 
$\gh$-class of $S$, and let 
$h \in H$ be an arbitrary element. Then:
\begin{enumerate}[(i)]
\item
\label{item-sch1}
$\mathrm{Stab}(H) = \{ t \in T^1 : ht \in H  \}$.
\item 
\label{item-sch2}
$\gamma(H) = \{ (u,v) \in \mathrm{Stab}(H) \times \mathrm{Stab}(H) : hu = hv  \}$.
\item
\label{item-sch3}
$H = h \mathrm{Stab}(H)$.
\item
\label{item-sch4}
If $H'$ is an $\gh$-class belonging to the same $\gl$-class of $S$ as $H$ then $\mathrm{Stab}(H) = \mathrm{Stab}(H')$ and 
$\Gamma(H) = \Gamma(H')$. 
\item 
\label{item-sch6}
If $H'$ is an $\gh$-class of $S$ belonging to the same $\gr$-class as $H$ then $\Gamma(H') \cong \Gamma(H)$.
\end{enumerate}
\end{proposition}

\section{The Rewriting Lemma}
\label{sec_rewriting}

The aim of this section is to prove a rewriting lemma which arises naturally from the theory of relative Green's relations, and which will be a vital tool for the proofs of many of the results about finiteness conditions that follow.

Throughout this section $S$ will be a semigroup and $T$ will be a subsemigroup of $S$.
We let $\{ H_i : i \in I \}$ be the set of relative $\gh$-classes in $S \setminus T$, with a fixed set of representatives $h_i \in H_i \ (i \in I)$, and relative Sch\"{u}tzenberger groups $\Gamma_i = \Gamma(H_i) = \mathrm{Stab}_T (H_i) / \gamma_i$. Set $I^1 = I \cup \{ 1 \}$ where we assume $1 \not\in I$. We introduce the convention $H_1 = \{ 1 \}$ and $h_1 = 1$ where $1$ is the external identity adjoined to $S$.

Next we introduce two mappings
\[
\rho: S^1 \times I^1 \rightarrow I^1, \quad \quad \lambda: I^1 \times S^1 \rightarrow I^1
\]
which reflect the way that the elements of $S^1$ act on the representatives $h_i$:

\begin{eqnarray}
\rho(s,i) & = & \begin{cases}
j & \mbox{if $sh_i \in H_j$} \\
1 & \mbox{if $sh_i \in T$},
\end{cases} \label{(1.4)}
\end{eqnarray}
and
\begin{eqnarray}
\lambda(i,s) & = & \begin{cases}
j & \mbox{if $h_i s \in H_j$} \\
1 & \mbox{if $h_i s \in T$}.
\end{cases} \label{(1.5)}
\end{eqnarray}

The following lemma introduces related elements $\sigma(s,i)$ and $\tau(i,s)$ which `connect'
$s h_i$ and $h_i s$ to their respective $\gh$-class representatives.

\begin{lemma}\label{lem_sigmaandtau}
For all $i \in I^1$ and $s \in S^1$ there exist $\sigma(s,i), \tau(i,s) \in T^1$ satisfying:
\begin{eqnarray}
s h_i & = & h_{\rho(s,i)} \sigma(s,i),  \label{(1.2)}
\end{eqnarray}
and
\begin{eqnarray}
h_i s & = & \tau(i,s) h_{\lambda(i,s)}. \label{(1.3)}
\end{eqnarray}
\end{lemma}
\begin{proof}
If $\rho(s,i) \neq 1$ we have $sh_i \gh h_{\rho(s,i)}$ and so there exists $\sigma(s,i) \in T^1$ satisfying
\[
sh_i = h_{\rho (s,i)} \sigma(s,i).
\]
Otherwise $\rho (s,i) = 1$, and setting $\sigma(s,i) = sh_i \in T^1$ equality \eqref{(1.2)} holds trivially.
The existence of $\tau(i,s)$ is proved dually. 
\end{proof}

The following lemma describes the effect of pushing an $\gh$-class representative through a product of elements of $S$ from left to right. 

\begin{lemma}[Rewriting lemma] \label{lemma_bookkeeping}
Let $i \in I^1$ and let $s_1, s_2, \ldots, s_n \in S$. Then
\begin{equation}
h_i s_1 s_2 \ldots s_n = t_1 t_2 \ldots t_n h_j \label{(5)}
\end{equation}
where $t_1, \ldots, t_n \in T^1$ and $j \in I^1$ are obtained as a result of the following recursion:
\begin{alignat}{3}
&i_1 && =  i&\quad& \label{(6)} \\
&i_{k+1} && =  \lambda (i_k, s_k) && (k = 1, \ldots, n), \label{(7)} \\
&j && =  i_{n+1} &&\label{(8)} \\
&t_k && =  \tau(i_k, s_k) && (k = 1, \ldots, n). \label{(9)} 
\end{alignat}
Furthermore:
\begin{enumerate}[(i)]
\item If all $s_q \in T$ and $h_i s_1 s_2 \ldots s_n \not\in T$ then 
$
h_i s_1 s_2 \ldots s_n \gl h_j.
$ 
\item If all $s_q \in T$ and $h_i s_1 s_2 \ldots s_n \in T$ then $j=1$ and so $h_j = h_1 =1$. 
\item If all $s_q \in T$ and $h_i s_1 s_2 \ldots s_n \gr h_i$ then 
$
h_i s_1 s_2 \ldots s_n \gh h_j. 
$
\end{enumerate}
\end{lemma}

\begin{lem_dual}[Dual rewriting lemma] \label{lemma_bookkeeping'}
\it
Let $i \in I^1$ and let $s_1, s_2, \ldots, s_n \in S$. Then
\begin{equation}
 s_1 s_2 \ldots s_n h_i = h_j t_1 t_2 \ldots t_n  \label{(5')}
\end{equation}
where $t_1, t_2, \ldots, t_n \in T^1$ and $j \in I^1$ are obtained as a result of the following recursion:
\begin{eqnarray}
i_n & = & i \label{(6')} \\
i_{k-1} & = & \rho (s_k, i_k) \quad (k = n, \ldots, 1), \label{(7')} \\
j & = & i_{0} \label{(8')} \\
t_k & = & \sigma(s_k, i_k) \quad (k = n, \ldots, 1). \label{(9')}
\end{eqnarray}
Furthermore:
\begin{enumerate}[(i)]
\item If all $s_q \in T$ and $s_1 s_2 \ldots s_n h_i \not\in T$ then
$
s_1 s_2 \ldots s_n h_i \gr h_j.
$
\item If all $s_q \in T$ and $s_1 s_2 \ldots s_n h_i \in T$ then $j=1$ and so $h_j = h_1 =1$.
\item If all $s_q \in T$ and $s_1 s_2 \ldots s_n h_i \gl h_i$ then
$
s_1 s_2 \ldots s_n h_i \gh h_j.
$
\end{enumerate}
\end{lem_dual}

\begin{proof}
We just prove Lemma~\ref{lemma_bookkeeping}. Lemma~\ref{lemma_bookkeeping}' may be proved using a dual argument.

For the first part we proceed by induction on $n$. The result holds trivially when $n=0$. Supposing that the result holds for $n$, the inductive step is as follows:
\[
\begin{array}{rcll}
h_i s_1 s_2 \ldots s_n s_{n+1} & = &
t_1 \ldots t_n h_{i_{n+1}} s_{n+1} & \quad \mbox{(by induction)} \\
& = &
t_1 \ldots t_n \tau(i_{n+1}, s_{n+1}) h_{\lambda(i_{n+1}, s_{n+1})} & \quad \mbox{(by \eqref{(1.3)})} \\
& = &
t_1 \ldots t_n t_{n+1} h_{i_{n+2}}.
\end{array}
\]

(i) We prove the result by induction on $n$. When $n=0$ there is nothing to prove. Now suppose that the result holds for $n-1$.
Because $s_n \in T$ and $h_i s_1 \ldots s_n \not\in T$ it follows that $h_i s_1 \ldots s_{n-1} \not\in T$ so we may apply induction to obtain:
\[
h_i s_1 s_2 \ldots s_{n-1} \gl h_{i_n}.
\]
This implies 
\[
h_i s_1 \ldots s_{n-1} s_n \gl h_{i_n} s_n \gh h_{\lambda (i_n,s_n)} = h_{i_{n+1}}
\]
by \eqref{(1.5)} and \eqref{(7)}. 

(ii) If $i=1$ then from \eqref{(1.5)}, \eqref{(6)}, \eqref{(7)} and \eqref{(8)} it follows that $1 = i_1 = i_2 = \ldots = i_{n+1} = j$.
Otherwise, since $h_i s_1 \ldots s_n \in T$ there exists $0 \leq k \leq n-1$ such that
\[
h_i s_1 \ldots s_k \not\in T \quad \& \quad h_i s_1 \ldots s_k s_{k+1} \in T.
\]
By (i) applied to $h_i s_1 \ldots s_k$ we obtain 
\[
h_{i_{k+1}} \gl h_i s_1 \ldots s_k
\]
which implies 
\[
h_{i_{k+1}} s_{k+1} \gl h_i s_1 \ldots s_k s_{k+1}
\]
and so, $h_{i_{k+1}} s_{k+1} \in T$. Hence by \eqref{(1.5)} it follows that $i_{k+2} = \lambda(i_{k+1}, s_{k+1}) = 1$. 
Then as above \eqref{(7)} gives $1 = i_{k+2} = i_{k+3} = \ldots = i_{n+1} = j$.

(iii) Again we proceed by induction on $n$. There is nothing to prove when $n=0$. Suppose that the result holds for $n-1$.
Since $h_i s_1 \ldots s_n \gr h_i$ there exists $t \in T$ such that $h_i s_1 \ldots s_n t = h_i$. But since $s_n \in T$ 
and $s_1 \ldots s_{n-1} \in T$
it follows that $h_i s_1 \ldots s_{n-1} \gr h_i$ and so we may apply induction. This gives
\[
h_i s_1 \ldots s_{n-1} \gh h_{i_n}.
\]
Since $h_i s_1 \ldots s_{n-1} \gr h_i s_1 \ldots s_{n-1} s_n$, by Proposition~\ref{basicproperties}\ref{prop_{Rel_Greens_Lemma}} the mapping $x \mapsto xs_n$ sends the $\gh$-class of $h_i s_1 \ldots s_{n-1}$ bijectively onto the $\gh$-class of $h_i s_1 \ldots s_{n-1} s_n$. In particular
\[
h_{i_n} s_n \gh h_i s_1 \ldots s_{n-1} s_n.
\]
On the other hand, $h_{i_n} s_n \in H_{\lambda(i_n,s_n)} = H_{i_{n+1}}$ by \eqref{(1.5)} and \eqref{(7)}, and so 
\[
h_{i_{n+1}} \gh h_{i_n} s_n \gh h_i s_1 \ldots s_n,
\]
as required. 
\end{proof}

\section{Generators for Subsemigroups}
\label{secsubsemigroups}

Let $S$ be a semigroup, $T$ be a subsemigroup of $S$ and $\{ H_i : i \in I \}$ the set of relative $\gh$-classes in $S \setminus T$. In this section we show how to relate generating sets for $S$, $T$ and the relative Sch\"{u}tzenberger groups $\Gamma(H_i)$. Throughout the section we use the same notation and conventions introduced in Section~\ref{sec_rewriting}.

If $B$ is a generating set for $T$ and $C$ is a subset of $S$ satisfying $S^1 = C^1 T^1$ then obviously $B \cup C$ generates $S$. In particular we have the following easy result.

\begin{theorem}\label{thm_smalltobig}
Let $S$ be a semigroup and let $T$ be a subsemigroup of $S$. If $B$ is a generating set for $T$ and $C = \{ h_i : i \in I  \}$ is a set of representatives of the relative $\gh$-classes of $S \setminus T$, then $B \cup C$ is a generating set for $S$. In particular, if $T$ is finitely generated and has finite Green index in $S$ then $S$ is finitely generated.
\end{theorem}

\begin{remark}
Of course, in the above theorem we can replace $C$ by a transversal of just the relative $\gr$-classes (or $\gl$-classes) in $S \setminus T$, and $B \cup C$ will still generate $S$.
\end{remark}

Now we go on to consider the more interesting converse statement. We begin by fixing a particular choice of $\sigma$ and $\tau$ from Lemma~\ref{lem_sigmaandtau}.

The following result provides a common generalisation of the classical result of Schreier for groups (see \cite[Chapter II]{lyndon_cgt} for example) and the analogous theorem for subsemigroups of finite Rees index due to Jura \cite{Jura1}.

\begin{theorem} \label{thm_finitegeneration}
Let $S$ be a semigroup generated by $A \subseteq S$, let $T$ be a subsemigroup of $S$, and let $I$, $\sigma$, $\tau$ be as above. Then $T$ is generated by the set
\[
B = \{  
\tau(i, \sigma(a,j)) : i,j \in I^1, \; a \in A
\}.
\]
In particular, if $S$ is finitely generated and $T$ has finite Green index in $S$, then $T$ is finitely generated. 
\end{theorem}
\begin{proof}
Let $t \in T$ and write $t = a_1 a_2 \ldots a_n$, a product of generators from $A$. Applying Lemma~\ref{lemma_bookkeeping}' gives
\[
t = h_{i_0} \sigma(a_1, i_1) \sigma(a_2, i_2) \ldots \sigma(a_n,i_n)
\]
where
\[
i_n = 1, \quad i_{k-1} = \rho(a_k, i_k), \quad k=n,n-1, \ldots, 1.
\]
This rewriting may be viewed as pushing the representative $h_1 = 1$ through the product from right to left using Lemma~\ref{lemma_bookkeeping}'. Note that $i_0$ is not necessarily equal to $1$ here, but if it were then we would be done since $\sigma(a_k,i_k) = \tau(1, \sigma(a_k, i_k)) \in B$. Applying Lemma~\ref{lemma_bookkeeping} we now perform an analogous rewriting pushing the representative $h_{i_0} = h_{j_1}$ back through the product from left to right giving
\[
\begin{array}{rcll}
& & h_{j_1} \sigma(a_1, i_1) \sigma(a_2, i_2) \ldots \sigma(a_n, i_n) &   \\
& = &
\tau(j_1, \sigma(a_1, i_1)) \tau(j_2, \sigma(a_2, i_2)) \ldots \tau(j_n, \sigma(a_n, i_n)) h_{j_{n+1}},
\end{array}
\]
where 
\[
j_1 = i_0, \quad j_{k+1} = \lambda(j_k, \sigma (a_k, i_k)), \quad k = 1,2,  \ldots, n. 
\]
Now by Lemma~\ref{lemma_bookkeeping}(ii)  since each $\sigma(a_k, i_k) \in T$ and 
\[
h_{j_1} \sigma(a_1, i_1) \sigma(a_2, i_2) \ldots \sigma(a_n, i_n) \in T
\]
it follows that  $j_{n+1} = 1$ and therefore
\[
t = \tau(j_1, \sigma(a_1, i_1)) \tau(j_2, \sigma(a_2, i_2)) \ldots \tau(j_n, \sigma(a_n, i_n)) \in \lb B \rb. \] 
The last statement in the theorem follows since if $A$ and $I$ are both finite then $B$ is finite.
\end{proof}

One natural question we might ask at this point is whether Theorem~\ref{thm_finitegeneration} might be
proved under
the weaker assumption that $S \setminus T$ is a union of finitely many $\gr$-classes (or dually $\gl$-classes).
Such a weakening is possible, for example, in the case of
groups (and more generally inverse semigroups) where for the complement
the properties of having
finitely many relative $\gr$-, $\gl$- or $\gh$-classes are all equivalent conditions.
The following example (and its dual) shows that for arbitrary semigroups
such a weakening of the hypotheses is not possible.

\begin{example}
Let $S$ be the semigroup, with a zero element $0$ and an identity $1$, defined by the following presentation:
\[
\lb a,b, b^{-1}, c \ | \ a^2 = c^2 = 0, \ ba = b^{-1}a = ca = cb = cb^{-1} = 0, \ bb^{-1} = b^{-1}b = 1 \rb.
\]
It is easily seen that a set of normal forms for the elements of $S$ is:
\[
N = \{0 \} \cup \{ a^i b^j c^k : i,k \in \{0,1\}, j \in \mathbb{Z}  \}.
\]
From this it follows that this semigroup is isomorphic to the semigroup of triples $S = \mathbb{Z}_2 \times \mathbb{Z} \times \mathbb{Z}_2 \cup \{ 0 \} $ with multiplication:
\[
(u,v,w)(d,e,f) = \begin{cases}
(u,v+e,f) & \mbox{if} \ w=d=0 \\
0 & \mbox{otherwise}.
\end{cases}
\]
Clearly $S$ is generated by $A = \{ (1,0,0), (0,1,0), (0,0,1), (0,-1,0)  \}$. Now define:
\[
T = \{ (x,y,z) \in S : z \geq x \} \cup \{ 0 \},
\]
where $\{ 0 ,1 \}$ is ordered in the usual way $0 < 1$. So $T$ contains all triples except those of the form $(1,i,0)$. Let $(x_1,y_1,z_1), (x_2,y_2,z_2) \in T$ be arbitrary. Then
\[
(x_1,y_1,z_1)(x_2,y_2,z_2) =
\begin{cases}
(x_1,y_1 + y_2,z_2) & \mbox{if $z_1 = x_2 = 0$} \\
0 & \mbox{otherwise},
\end{cases}
\]
and in the first of these two cases $(x_1,y_1 + y_2,z_2) \in T$ since $z_2 \geq x_1 = 0$. It follows that $T$ is a subsemigroup of $S$.
Now $S \setminus T$ has a single relative $\gr$-class since $S \setminus T = \{ (1,i,0): i \in \mathbb{Z}  \}$ and
\[
(1,i,0)(0,j-i,0) = (1,j,0).
\]
On the other hand, $T$ is not finitely generated since the elements in the set $\{ (1,j,1) : j \in \mathbb{Z} \}$ cannot be properly decomposed in $T$, as:
\[
(x_1,y_1,z_1)(x_2,y_2,z_2) = (1,j,1)
\]
(where $(x_i,y_i,z_i) \neq (1,j,1)$) implies that $x_1=1$, $z_2=1$ and $z_1 = x_2 = 0$. But then $(x_1,y_1,z_1) = (1,y_1,0) \not\in T$ which is a contradiction.

In conclusion, $S$ is finitely generated, $S \setminus T$ has finitely many relative $\gr$-classes, but $T$ is not finitely generated.
\end{example}

Before giving the next example we introduce a construction which will be used
several times throughout the paper.
It is a special case of the well known \emph{strong semilattice of semigroups},
where  the underlying semilattice is just a $2$-element chain; see \cite[Chapter~4]{howie_fundamentals} for details of the general construction.

\begin{definition}
\label{def_strong_semilattice}
Let $T$ and $U$ be semigroups and let $\phi: T \rightarrow U$ be a homomorphism. From this triple we construct a monoid $S = \mathcal{S}(T,U,\phi)$ where $S = T \dotcup U$ and multiplication is defined in the following way. Given $x,y \in S$ if $x,y \in T$ then we multiply as in $T$; if $x,y \in U$ then we multiply as in $U$; if $x \in T$ and $y \in U$ then take the product of $\phi(x)$ and $y$ in $U$; if $x \in U$ and $y \in T$ then take the product of $x$ and $\phi(y)$ in $T$.
\end{definition}

Another natural way that one might consider weakening the hypotheses of
Theorem~\ref{thm_finitegeneration} would be to replace the condition that there are finitely many relative $\gh$-classes in $S \setminus T$ with the weaker property that there is a finite subset $C$ of $S$ such that
\begin{equation}
\forall s \in S,  \exists c \in C, \exists t, t' \in T : s = ct' = tc. \label{***}
\end{equation}
The following example shows that Theorem~\ref{thm_finitegeneration} cannot be proved under this weaker assumption.
\begin{example}
Let $M$ be a monoid finitely generated by a set $A$, and 
with a two-sided ideal $R$ and suppose that, as a two-sided ideal, $R$ is not finitely generated.
Such examples exist; for example we could take $M$ to be the free monoid on $\{ a, b \}$ and $R$ to be the two-sided ideal generated by all words of the form $ab^ia$ $(i \in \mathbb{N})$.
Let $\overline{M}$ be an isomorphic copy of $M$ with isomorphism:
\[
\phi: M \rightarrow \overline{M}, \quad m \mapsto \overline{m}.
\]
Define $S = \mathcal{S}(M,\overline{M},\phi)$ and $T = M \cup \overline{R}$ where $\overline{R} = \{ \overline{r} : r \in R \}$.

Then $S$ is finitely generated, by $A \cup \{ \overline{e} \}$ where $e$ is the identity of $M$, and $T$ is a subsemigroup of $S$. Also $T \leq S$ satisfies condition $\eqref{***}$ with $C = \{ e, \overline{e} \}$, since for all $s \in S$
\[
s = \begin{cases}
es = se & \mbox{if $s \in M \subseteq T$} \\
\overline{e} m = m \overline{e} & \mbox{if $s = \overline{m}$ for some $m \in M \subseteq T$.}
\end{cases}
\]
However, $T$ is not finitely generated. Indeed, if $T$ were finitely generated then there would be a finite subset $X$ of $R$ satisfying $T = \lb M \cup \overline{X} \rb$. Then for every $r \in R$ we could write $\overline{r} \in T$ as a product of elements of $M \cup \overline{X}$ where, since $M \leq T$, this product would need to have at least one term from $\overline{X}$. Thus we would have $\overline{r} = \alpha \overline{x} \beta$ for some $x \in X$ and $\alpha, \beta \in T^1$ and applying $\phi^{-1}$ it would follow that, in $M$, $X$ generates $R$ as a two-sided ideal. Since $X$ is finite, this would contradict the original choice of $R$.

\end{example} 

\section{Generators for the Sch\"{u}tzenberger Groups}  
\label{secschgps}

As above, let $S$ be a semigroup and let $T$ be a subsemigroup of $S$. In this section we show how generating sets for the $T$-relative Sch\"{u}tzenberger groups in $S$ may be obtained from generating sets of $T$.

Fix an arbitrary relative $\gh$-class $H$ of $S$ and fix a representative $h \in H$. We do not insist here that $H$ is a subset of the complement $S \setminus T$, and thus allow the possibility that $H \subseteq T$ (meaning that $H$ is just an $\gh$-class of $T$ in the classical sense). Let $\mathrm{Stab}(H) \leq T$ be the stabilizer of $H$, let $\gamma$ be the Sch\"{u}tzenberger congruence and $\Gamma = \mathrm{Stab}(H) / \gamma$ be the corresponding relative Sch\"{u}tzenberger group. Let $\{H_{\lambda}: \lambda \in \Lambda\}$ be the collection of all $\gh$-classes in the $\gr$-class of $H$. By Proposition~\ref{basicproperties}(ii) we can choose elements $p_{\lambda}, p_{\lambda}' \in T^1$ such that
\[
H p_{\lambda} = H_{\lambda}, \quad h_1 p_{\lambda} p_{\lambda}' = h_1, \quad h_2 p_{\lambda}' p_{\lambda} = h_2, \quad (\lambda \in \Lambda, \ h_1 \in H, \ h_2 \in H_{\lambda}).
\]
Also we assume that $\Lambda$ contains a distinguished element $\lambda_1$ with
\[
H_{\lambda_1} = H, \quad  p_{\lambda_1} = p_{\lambda_1}' = 1.
\]
We can define an action of $T^1$ on the set $\Lambda \cup \{ 0 \}$ by:
\[
\lambda \cdot t = \begin{cases}
\mu & \mbox{if} \ \lambda, \mu \in \Lambda \ \& \ H_{\lambda}t = H_{\mu} \\
0 & \mbox{otherwise}.
\end{cases}
\]
In the classical (non-relative) case generating sets for
Sch\"{u}tzenberger groups may be obtained from a generating set of the containing monoid by
adapting the classical method in group theory for computing Schreier generators (see \cite[Chapter II]{lyndon_cgt}) for a subgroup (this may be found implicitly in Sch\"{u}tzenberger's original papers \cite{Schutzenberger1957}, \cite{Schutzenberger1958}, and explicitly in \cite{Ruskuc2}). In the following we record the easy generalisation of that result to the relative setting (the original classical results may be obtained by setting $S = T$).
\begin{theorem}\label{thm_schutz_gen}
Let $S$ be a semigroup, let $T$ be a subsemigroup of $S$ generated by a set $B$, and let $H$ be an arbitrary $T$-relative $\gh$-class of $S$. Then the relative Sch\"{u}tzenberger group $\Gamma = \Gamma(H)$ of $H$ is generated by:
\[
X = \{ (p_{\lambda} b p_{\lambda \cdot b}') / \gamma: \lambda \in \Lambda, \ b \in B, \ \lambda \cdot b \neq 0  \}.
\]
In particular, if $T$ is finitely generated, and the relative $\gr$-class of $H$ contains only finitely many relative $\gh$-classes, then $\Gamma$ is finitely generated.
\end{theorem}
\begin{proof}
First we prove that with
\[
\Gamma' = \{ (p_{\lambda} t p_{\lambda \cdot t}') / \gamma: \lambda \in \Lambda, \ t \in T, \ \lambda \cdot t \neq 0 \}
\]
we have $\Gamma = \Gamma'$. On one hand, given $(p_{\lambda} t p_{\lambda \cdot t}') / \gamma \in \Gamma'$ since:
\[
H p_{\lambda} t p_{\lambda \cdot t}' = H_{\lambda} t p_{\lambda \cdot t}' = H_{\lambda \cdot t} p_{\lambda \cdot t}' = H
\]
it follows that $p_{\lambda} t p_{\lambda \cdot t}' \in \mathrm{Stab}(H)$, the stabilizer of $H$, and therefore $\Gamma'$ is well-defined and $\Gamma' \subseteq \Gamma$. On the other hand, given $v / \gamma \in \Gamma$ since $Hv = H$ it follows that $\lambda_1 \cdot v = \lambda_1$ and therefore that $v / \gamma = (p_{\lambda_1} v p_{\lambda_1}') / \gamma \in \Gamma'$, and $\Gamma \subseteq \Gamma'$.

To finish the proof we must show that an arbitrary element $g = (p_{\lambda} t p_{\lambda \cdot t}') / \gamma \in \Gamma'$ can be written as a product of generators from $X$. Write $t = b_1 \ldots b_m$ $(b_j \in B)$. We proceed by induction on $m$. If
$m=1$ we have $g \in X$. Now let $m>1$ and assume that the result holds for all smaller values. Let $a = b_1$ and $u = b_2 \ldots b_m$. Now we have: 
\[
\begin{array}{rcll}
g & = &  (p_{\lambda} t p_{\lambda \cdot t }') / \gamma	\\
& =	&
(p_{\lambda} au p_{\lambda \cdot au}') / \gamma & 	\\
& = &
(p_{\lambda} a p_{\lambda \cdot a}' p_{\lambda \cdot a}    u p_{(\lambda \cdot a) \cdot u}') / \gamma
& \mbox{(by definition of $p_{\lambda \cdot a}$)} \\
& = &
(p_{\lambda} a p_{\lambda \cdot a}') / \gamma \; (p_{\lambda \cdot a}    u p_{(\lambda \cdot a) \cdot u}') / \gamma
& \mbox{(since $p_{\lambda} a p_{\lambda \cdot a}', \; p_{\lambda \cdot a}    u p_{(\lambda \cdot a) \cdot u}' \in T$)} \\
& \in & \lb X \rb & \mbox{(by induction).}
\end{array}
\]
The last part of the theorem follows since if $B$ is finite and $\Lambda$ is finite then $X$ is finite. \end{proof}

Combining this result with Theorem~\ref{thm_finitegeneration} we obtain the following. 

\begin{theorem}\label{thm_main_fg}
Let $S$ be a semigroup, let $T$ be a subsemigroup of $S$ with finite Green index, and let $\{ H_i : i \in I \}$ be the $T$-relative $\gh$-classes in the complement $S \setminus T$. Then $S$ is finitely generated if and only if $T$ is finitely generated, in which case all the
relative Sch\"{u}tzenberger groups $\Gamma(H_i)$ are finitely generated as well.    
\end{theorem}

\section{Building a presentation from the subsemigroup and Sch\"{u}tzenberger groups}
\label{sec_presentations}

Given a semigroup $S$ and a subsemigroup $T$, in this section we show how one can obtain a presentation for $S$ in terms of a given presentation for $T$ and presentations for all the relative Sch\"{u}tzenberger groups of $S \setminus T$. In the case that the Green index of $T$ in $S$ is finite we shall see that finite presentability is preserved.

A (semigroup) \emph{presentation} is a pair $\mathfrak{P} = \lb A| \mathfrak{R} \rb$ where $A$ is a
an alphabet and $\mathfrak{R} \subseteq A^+ \times A^+$ is a set
of pairs of words. An element $(u,v)$ of $\mathfrak{R}$ is called a \emph{relation} and is
usually written $u=v$. We say that $S$ is the \emph{semigroup defined by the presentation}
$\mathfrak{P}$ if $S \cong A^+ / \eta$ where $\eta$ is the smallest congruence on $A^+$
containing $\mathfrak{R}$. We may think of $S$ as the largest semigroup
generated by the set $A$ which satisfies all the relations of $\mathfrak{R}$. We say that
a semigroup $S$ is \emph{finitely presented} if it can be defined by $\lb A|\mathfrak{R} \rb$ where $A$ and
$\mathfrak{R}$ are both finite.

Let $S$ be a semigroup defined by a presentation $\lb A|\mathfrak{R} \rb$, where we identify $S$ with $\A^+ / \eta$.  We say that the word $w \in A^+$ \emph{represents the element} $s \in S$ if $s = w / \eta$.
Given two words $w,v \in A^+$ we write $w = v$ if $w$
and $v$ represent the same element of $S$ and write $w \equiv v$ if $w$ and $v$ are
identical as words.

We continue to follow the same notation and conventions as in previous sections, so $S$ is a semigroup, $T$ is a subsemigroup, and $\Gamma_i = \mathrm{Stab}(H_i) / \gamma_i = \Gamma(H_i)$ $(i \in I)$ are the Sch\"{u}tzenberger groups of the $T$-relative $\gh$-classes in $S \setminus T$.
As above we also assume $1 \not\in I$ and follow the convention $H_1 = \{ 1 \}$ and $h_1=1$ where $1$ is the external identity adjoined to $S$.

Let $\lb B | Q \rb$ be a presentation for $T$ and $\beta: B^+ \rightarrow T$ be the natural homomorphism associated with this presentation (mapping each word to the element it represents). Next define $A = B \cup \{ d_i : i \in I \}$ and
extend $\beta$ to $\alpha: A^+ \rightarrow S$ given by extending the map
\[
\alpha(a) = \begin{cases}
\beta(a) & \mbox{if $a \in B$} \\
h_i & \mbox{if $a=d_i$ for some $i \in I$}
\end{cases}
\]
to a homomorphism.
It follows from Theorem~\ref{thm_smalltobig} that $\alpha$ is surjective.
We also
introduce the symbol $d_1$ which we use to denote the empty word.

For every $i \in I$ let $\lb C_i | W_i \rb$ be a (semigroup) presentation for the group $\Gamma_i$ and let $\xi_i: C_i^+ \rightarrow \Gamma_i$ be the associated homomorphism.
By Proposition~\ref{schutzstabproperties}(iv), for all $i,j \in I$ if $h_i \gl^T h_j$ then
$\mathrm{Stab}(H_i) = \mathrm{Stab}(H_j)$ and $\Gamma(H_i) = \Gamma(H_j)$.
Therefore we may suppose without loss of generality that
for all $i,j \in I$:
\begin{equation}
\label{dagger}
\begin{array}{lll}
h_i \gl^T h_j & \Rightarrow& C_i = C_j\ \&\ W_i = W_j\\
(h_i,h_j)\not\in\gl^T &\Rightarrow& C_i \cap C_j = \varnothing\ \&\ W_i \cap W_j = \varnothing.
\end{array}
\end{equation}
For every letter $c \in C_i$ ($i \in I$) we have
\[
\xi_i (c) \in \Gamma_i = \mathrm{Stab}(H_i) / \gamma_i.
\]
Since $\mathrm{Stab}(H_i) \subseteq T$ and $\beta: B^+ \rightarrow T$ is surjective there exists a word $\overline{\xi_i}(c) \in B^+$ with $\beta(\overline{\xi_i}(c) ) \in \mathrm{Stab}(H_i)$ and
\[
\beta(  \overline{\xi_i}(c)  ) / \gamma_i = \xi_i(c).
\]
This defines a family of mappings $\overline{\xi}_i : C_i \rightarrow B^+$ ($i \in I$), which when taken together define a
mapping from $C = \bigcup_{i \in I} C_i$ to $B^+$, which in turn extends uniquely to a homomorphism $\overline{\xi} : C^+ \rightarrow B^+$. For $i \in I$ define $\overline{\xi}_i = \overline{\xi} \upharpoonright_{C_i^+}$, the restriction of $\overline{\xi}$ to the set $C_i^+ \subseteq C^+$. 
Since $\beta$ and $\xi_i$ are homomorphisms, and $\gamma_i$ is a congruence, the mapping $\overline{\xi}_i$ satisfies:
\[
\beta(  \overline{{\xi}_i}(w)  ) / \gamma_i = \xi_i(w)
\]
for all $w \in C_i^+$.

In order to write down our presentation for $S$ we need to lift the mappings $\rho$, $\lambda$, $\sigma$ and $\tau$ introduced
in Section~\ref{sec_rewriting} from elements of $S$ to words, in the obvious way.
Abusing notation we shall use the same symbols for these liftings. Thus, considered as mappings on words,
we have
$$
\begin{array}{ll}
\rho: A^* \times I^1 \rightarrow I^1, & \lambda: I^1 \times A^* \rightarrow I^1,\\
\sigma: A^* \times I^1 \rightarrow B^*,&  \tau: I^1 \times A^* \rightarrow B^*,
\end{array}
$$
where
\[
\begin{array}{rclrcl}
\rho(w,i) & = & \begin{cases}
j & \mbox{if $\alpha(w) h_i \in H_j$} \\
1 & \mbox{if $\alpha(w) h_i \in T$},
\end{cases}  &
\lambda(i,w) & = & \begin{cases}
j & \mbox{if $h_i \alpha(w) \in H_j$} \\
1 & \mbox{if $h_i \alpha(w) \in T$},
\end{cases}
\\
\alpha(w) h_i & = & h_{\rho(w,i)} \alpha(\sigma(w,i)), &
h_i \alpha(w) & = & \alpha(\tau(i,w)) h_{\lambda(i,w)}.
\end{array}
\]

\begin{theorem}\label{thm_pres_smalltobig}
Suppose that $T$ is a subsemigroup of $S$, and that $\langle B\:|\: Q\rangle$ is a presentation for $T$.
With the remaining notation as above,
$S$ is defined by the presentation with generators $A = B \cup \{ d_i | i \in I  \}$
and set of defining relations $Q$ together with:
\begin{alignat}{3}
&ad_i & & = d_{\rho(a,i)} \sigma(a,i) &\quad&  (a \in A, i \in I^1),    \label{22} \\
& d_j b && = \tau(j,b) d_{\lambda(j,b)} &&  (b \in B, j \in I^1),  \label{33} \\
&d_i \bar{\xi} (u) && = d_i \bar{\xi} (v) &&  (i \in I^1, (u,v) \in W_i).  \label{44}
\end{alignat}
In particular if $T$ has finite Green index in $S$, and all of the Sch\"{u}tzenberger groups $\Gamma_i$ are finitely presented, then $S$ is finitely presented.
\end{theorem}

\begin{proof} The defining relations $Q$ and  \eqref{22}--\eqref{44} clearly all hold. We want to show that any relation $w_1 = w_2$ ($w_1, w_2 \in A^+$) that holds in $S$ is a consequence of these relations.

Consider the word $w_1$ and transform it using our defining relations as follows. First write $w_1 = w_1 d_1$. Then use relations \eqref{22} to move $d_1$ through the word $w_1$ from right to left, one letter at a time. We obtain a word $d_i w_1'$ where $w_1' \in B^+$ and the subscript $i$ is computed by the algorithm given in Lemma~\ref{lemma_bookkeeping}'.
Next, use relations \eqref{33} to move $d_i$ through $w_1'$ from left to right, one letter at a time, to obtain a word $w_1'' d_j$ where $w_1'' \in B^+$ and $d_j$ is computed by the algorithm given in Lemma~\ref{lemma_bookkeeping}.

If $\alpha(w_1) \in T$ we have $j=1$ by Lemma~\ref{lemma_bookkeeping}(ii), and so we have transformed $w_1$ into a word $w_1'' \in B^+$. The same process applied to $w_2$ would then give a word $w_2'' \in B^+$. Since $\lb B | Q \rb$ is a presentation for $T$, the relation $w_1'' = w_2''$ is a consequence of $Q$, and so $w_1 = w_2$ is a consequence of the relations in this case.

Now consider the case $\alpha(w_1) = \alpha(w_2) \not\in T$. In this case, applying Lemma~\ref{lemma_bookkeeping}(i) shows that $h_j = \alpha(d_j) \gl \alpha(w_1)$. Using relations \eqref{22} once more, we rewrite $w_1'' d_j$ into $d_k w_1'''$. This time Lemma~\ref{lemma_bookkeeping}'(iii) applies, and so $h_k = \alpha(d_k) \gh \alpha(w_1)$. Furthermore, because $\alpha(d_j) \gl \alpha(w_1) \gh \alpha(d_k)$, it follows that all the intermediate $d_l$ appearing in this rewriting also satisfy $\alpha(d_l) \gl \alpha(w_1)$, and so $C_l = C_k$ by \eqref{dagger}. Thus all $\sigma(b,l)$ arising from applications of \eqref{22} are in the image of $\bar{\xi}_k$, and, since $\bar{\xi}_k$ is a homomorphism it follows that $w_1''' \equiv \bar{\xi}_k(\overline{w}_1) \equiv \bar{\xi} (\overline{w}_1)$ for some $\overline{w}_1 \in C_k^+$. The same process applied to $w_2$ rewrites it into a word $d_r \bar{\xi} (\overline{w}_2)$. From 
$$
h_1 = \alpha(d_r) \gh \alpha(w_2) = \alpha(w_1) \gh \alpha(d_k) = h_k
$$ 
it follows that $r=k$, and $\overline{w}_2 \in C_k^+$.

From $\alpha(w_1) = \alpha(w_2)$ we have $h_k \alpha( \bar{\xi} (\overline{w}_1)) = h_k \alpha (\bar{\xi} (\overline{w}_2))$, and so 
$$( \alpha(\bar{\xi}(\overline{w}_1)), \alpha(\bar{\xi}(\overline{w}_2)) ) \in \gamma_k.$$ Since $\lb C_k | W_k \rb$ is a presentation
for $\Gamma_k$, it follows that $\overline{w}_1 = \overline{w}_2$ is a consequence of the relations $W_k$. So, $\overline{w}_2$ can be obtained from $\overline{w}_1$ by applying relations from $W_k$. We shall now show that this can be translated into a sequence of applications of the relations \eqref{33} and \eqref{44} transforming $d_k \bar{\xi} (\overline{w}_1)$ into $d_k \bar{\xi} (\overline{w}_2)$.

Clearly it is sufficient to consider the case where $\overline{w}_2$ is obtained from $\overline{w}_1$ by a single application of a relation from $W_k$, so:
\[
\overline{w}_1 \equiv xuy, \quad \overline{w}_2 \equiv xvy, \quad x,y \in C_k^*, \ (u=v) \in W_k.
\]
There is a sequence of applications of \eqref{33} transforming $d_k \bar{\xi}(x)$ into $zd_t$ where $z \in B^*$. Moreover, since $x \in C_k^+$, it follows that $\alpha(\bar{\xi}(x)) \in \mathrm{Stab}(H_k)$ and so $$\alpha(d_k \bar{\xi} (x)) \gh \alpha (d_k),$$ 
implying $t=k$. Now applying \eqref{44} we obtain:
$$
d_k \bar{\xi} (\overline{w}_1) \equiv d_k \bar{\xi} (x) \bar{\xi} (u) \bar{\xi} (y) = z d_k \bar{\xi} (u) \bar{\xi}(y)
= z d_k \bar{\xi} (v) \bar{\xi} (y) = d_k \bar{\xi} (x) \bar{\xi} (v) \bar{\xi} (y) \equiv d_k \bar{\xi} (\overline{w}_2),
$$
thus completing the proof of the theorem.
\end{proof}

At present we do not know how to obtain `nice' presentations in the converse direction.
In particular, we pose:

\begin{question}
\label{presentationsquestion}
Let $T$ be a subsemigroup of finite Green index in a semigroup $S$.
Supposing that $S$ is finitely presented, is it true that:
(i) $T$ is necessarily finitely presented?
(ii) All $T$-relative Sch\"{u}tzenberger groups of $\gh^T$-classes in $S\setminus T$ are necessarily finitely presented?
\end{question} 

If the answers are affirmative, the proof is likely  to involve a combination of the methods used in
the classical Reidemeister--Schreier theory for groups, those for Rees index \cite{Ruskuc1}, and Sch\"{u}tzenberger groups \cite{Ruskuc2}. 
A major obstacle at present is the nature of the rewriting process employed in the proof of Theorem 
\ref{thm_main_fg}, whereby a word is first rewritten from left to right, and then once again from right to left.
This is in contrast with the rewritings employed in all the other contexts mentioned above, which are all essentially `one-sided'.

In the remainder of this section we give some corollaries, examples and further comments concerning Theorem \ref{thm_pres_smalltobig}

To begin with, note that Theorem~\ref{thm_pres_smalltobig} applies when
the complement is finite, in which case all of the relative Sch\"{u}tzenberger groups $\Gamma_i$ are finite and hence finitely presented, so we recover the following result, originally proved in \cite{Ruskuc1}.

\begin{corollary}[{\cite[Theorem~4.1]{Ruskuc1}}]
Let $S$ be a semigroup and let $T$ be a subsemigroup of $S$ with finite Rees index. If $T$ is finitely presented then $S$ is finitely presented.
\end{corollary}

In Example \ref{ex_fpcounter} and Theorems \ref{thm_aptrick1}, \ref{thm_aptrick2} below we will make use of the construction $\mathcal{S}(U,V,\phi)$, introduced in Definition~\ref{def_strong_semilattice}.
But first we record the following properties of this construction; the proofs are straightforward and are omitted:

\begin{lemma}
\label{lem_free_strong_semilattice}
Let $\phi\::\: T\rightarrow U$ be a surjective homomorphism of semigroups, and let $S=\mathcal{S}(T,U,\phi)$.
\begin{enumerate}[(i)]
\item
$T\leq S$ and $S\setminus T=U$.
\item
The relative $\gr^T$-classes, $\gl^T$-classes and $\gh^T$-classes in $U$ are precisely
$\gr$-classes, $\gl$-classes and $\gh$-classes respectively of $U$.
\item
The $T$-relative Sch\"{u}tzenberger group of an $\gh^T$-class $H\subseteq U$ is isomorphic to
the Sch\"{u}tzenberger group of $H$.
\end{enumerate}
\end{lemma}

We now proceed to exhibit an example which shows that the the condition of finite presentability on the relative Sch\"{u}tzenberger groups in Theorem~\ref{thm_pres_smalltobig} cannot be dropped.

\begin{example}
\label{ex_fpcounter}
Let $G$ be a finitely presented group which has a non-finitely presented homomorphic image $H$,
and let $\phi\::\: G\rightarrow H$ be an epimorphism.
($H$ can be chosen to be any finitely generated, non-finitely presented group, say with $r$ generators, and $G$ to be free of rank $r$.)
Let $S=\mathcal{S}(G,H,\phi)$.
By Lemma \ref{lem_free_strong_semilattice}, $G$ has Green index 2 in $S$.
On the other hand, $S$ is not finitely presented.
To see this one can check the easy facts that $H$ is a retract of $S$, and that finite presentability is
preserved under retracts (see also \cite{Wang2000}).
Alternatively one can apply results on strong semilattices of monoids from
\cite{Araujo2001}.
\end{example}

As another application of Theorem~\ref{thm_pres_smalltobig}, we obtain a rapid proof of the following result from \cite{Ruskuc2}:

\begin{theorem}[{\cite[Corollary~3.3]{Ruskuc2}}]
\label{thm_aptrick1}
Let $S$ be a semigroup with finitely many left and right ideals. If all Sch\"{u}tzenberger groups of $S$ are finitely presented then so is $S$.
\end{theorem}
\begin{proof}
Let $\{ H_i: i \in I \}$ be the set of all $\gh$-classes of $S$ where for each $i \in I$, $h_i \in H_i$ is a fixed representative and  $\Gamma_i = \Gamma(H_i)$ denotes the  Sch\"{u}tzenberger  group of $H_i$.
Suppose that all the Sch\"{u}tzenberger groups of $S$ are finitely presented. In particular they are all finitely generated and from this it easily follows that $S$ itself is finitely generated. Indeed, for each $i \in I$ we may fix a finite subset $A_i$ of $\mathrm{Stab}(H_i)$ satisfying $\lb A_i / \gamma_i \rb = \Gamma_i$. Then it is easily seen that
\[
A = (\bigcup_{i \in I} A_i) \cup \{ h_i : i \in I \}
\]
is a finite generating set for $S$.

Now let $W = \mathcal{S}(F,S,\phi)$ where $F$ is an appropriate free semigroup of finite rank.
Since $S$  has only finitely many $\gh^S$-classes and all the Sch\"{u}tzenberger groups are finitely presented, by Lemma~\ref{lem_free_strong_semilattice} if follows that $F$ is a subsemigroup of $W$ with finite Green index and with all the $F$-relative  Sch\"{u}tzenberger groups of $\gh$-classes in $W\setminus F=S$ finitely presented.
Since $F$ is free of finite rank, and hence is finitely presented, it follows from Theorem~\ref{thm_pres_smalltobig} that
$W$ is finitely presented. As in Example~\ref{ex_fpcounter} this implies that $S$ is finitely presented, since $S$ is a retract of $W$.
\end{proof}

We end this section by observing that the same trick used in the previous theorem may be applied to recover the corresponding result (originally proved in \cite{GrayRuskucResFin}) for residual finiteness, by using the
following result from \cite{gray_green1}:

\begin{proposition}[{\cite[Theorem 20]{gray_green1}}]
\label{schgprf}
Suppose $T$ is a subsemigroup of finite Green index in a semigroup $S$.
Then $S$ is residually finite if and only if $T$ and all the $T$-relative Sch\"{u}tzenberger
groups of $S\setminus T$ are residually finite.
\end{proposition}

Recall that a semigroup $S$ is \emph{residually finite} if for any pair $x, y \in S$ of distinct elements
there exists a homomorphism $\phi$ from $S$ into a finite semigroup such that $x \phi \neq y \phi$.
Clearly this is equivalent to the existence of a congruence with finitely many classes separating $x$ from $y$.

\begin{theorem}[{\cite[Theorem 7.2]{GrayRuskucResFin}}]
\label{thm_aptrick2}
Let $S$ be a semigroup with finitely many left and right ideals. Then $S$ is residually finite if and only if all of the Sch\"{u}tzenberger groups of $S$ are residually finite.
\end{theorem}

\begin{proof}
Let $\phi\::\: F\rightarrow S$ be an epimorphism from a (not necessarily finitely generated this time) free semigroup onto $S$,
and let $W=\mathcal{S}(F,S,\phi)$.
It is not hard to see that $W$ is residually finite if and only if
$S$ is residually finite. The direct part  of this claim is trivial since $S$ is a subsemigroup of $W$. For the converse, given $x,y \in W$ with $x \neq y$ we have the following possibilities: If $x \in F$ and $y \in S$ (or vice versa) then the congruence with two classes $F$ and $S$ separates $x$ from $y$. If $x,y \in F$ then we may identify all the elements in $S$ and apply the fact that $F$ is residually finite to separate $x$ from $y$ with a finite index congruence.
Finally, if $x,y \in S$ then since $S$ is residually finite there is a finite index congruence $\sigma$ on $S$ separating $x$ from $y$, and this may be extended to a finite index congruence on $W$ by taking the preimage of $\sigma$ under $\phi$, completing the proof of our assertion.

On the other hand since $F$ has finite Green index in $W$, and $F$ is residually finite, it follows from Proposition \ref{schgprf} that $W$ is residually finite if and only if all of the $F$-relative Sch\"{u}tzenberger groups of $\gh^F$-classes in $S$ are residually finite. But by Lemma~\ref{lem_free_strong_semilattice} these are precisely the Sch\"{u}tzenberger groups of $S$, and this completes the proof of the theorem.
\end{proof}

\section{Malcev presentations}
\label{secmalcev}

In the previous section we outlined the difficulties, related to the specific nature of our rewriting process, that at present prevent us from proving
that finite presentability is preserved when passing to subsemigroups of finite Green index.
In this section we prove such a result for so-called Malcev presentations, which are presentations of 
semigroups that can be embedded into groups.
(For a survey of the theory of Malcev presentations, see \cite{c_survey}.)
We do this by dispensing with rewriting altogether, and using properties of universal groups instead.

A congruence $\sigma$ on a semigroup $S$ is said to be a
\defterm{Malcev congruence} if $S/\sigma$ is embeddable in a group.
If $\{\sigma_i : i \in I\}$ is a set of Malcev congruences on $S$,
then $\sigma = \bigcap_{i \in I} \sigma_i$ is also a Malcev congruence
on $S$. This is true because $S/\sigma_i$ embeds in a group $G_i$ for
each $i \in I$, so $S/\sigma$ embeds in $\prod_{i \in I} S/\sigma_i$,
which in turn embeds in $\prod_{i \in I} G_i$.

Let $A^+$ be a free semigroup; let $\rho \subseteq A^+ \times A^+$ be
any binary relation on $A^+$. Let $\rho^M$ denote the smallest
Malcev congruence containing $\rho$ --- namely,
\[
\rho^M = \bigcap \left\{\sigma : \sigma \supseteq \rho, \text{
  $\sigma$ is a Malcev congruence on }A^+\right\}.
\]
Then $\langle A \:|\: \rho\rangle$ is a
\defterm{Malcev presentation} for (any semigroup isomorphic to)
$A^+/\rho^M$.

The main result of this section (generalising \cite[Theorem 1]{crr_finind}) is:

\begin{theorem}
\label{thmmalcev1}
Let $S$ be a group-embeddable semigroup, and let $T$ be a subsemigroup of finite Green index in $S$. 
Then $S$ has a finite Malcev presentation if and
only if $T$ has a finite Malcev presentation.
\end{theorem}

The proof of Theorem \ref{thmmalcev1} is at the end of the section.
We begin by recalling
the concept of universal groups of semigroups and
their connection to Malcev presentations. For further background on
universal groups  refer to \cite[Chapter~12]{clifford_semigroups2};
for their interaction with Malcev presentations, see
\cite[\S1.3]{c_phdthesis}.

Let $S$ be a group-embeddable semigroup. The \defterm{universal group}
$U$ of $S$ is the largest group into which $S$ embeds and which $S$
generates, in the sense that all other such groups are homomorphic
images of $U$.
The concept of a universal group can be defined for all semigroups, not
just those that are group-embeddable. However, the definition above
will suffice for the purposes of this paper. The universal group of a
semigroup is unique up to
isomorphism.

\begin{proposition}[{\cite[Construction~12.6]{clifford_semigroups2}}]
\label{prop:ugsamepres}
Let $S$ be a group-embeddable semigroup. Suppose $S$ is presented by (an ordinary semigroup presentation) 
$\langle A\:|\: \rho\rangle$
for some alphabet $A$ and set of defining relations $\rho$. Then
the group defined by the presentation $\langle A\:|\: \rho\rangle$ is {\rm[}isomorphic to{\rm]} the universal group of $S$.
\end{proposition}

The following two results show the connection between universal groups
and Malcev presentations. The proof of the first result is somewhat long and
technical; the second is a fairly direct corollary of the first.

\begin{proposition}[{\cite[Proposition~1.3.1]{c_phdthesis}}]
\label{prop:sgmpresiffugpres}
Let $S$ be a semigroup that embeds into a group. If
$\langle A\:|\: \rho\rangle$ is a Malcev presentation for $S$, then the universal
group of $S$ is presented by $\langle A\:|\: \rho\rangle$ considered as a group presentation. Conversely, if
$\langle A\:|\: \rho\rangle$ is a presentation for the universal group of $S$,
where $A$ represents a generating set for $S$ and $\rho \subseteq A^+
\times A^+$, then $\langle A\:|\: \rho\rangle$ is a Malcev presentation for $S$.
\end{proposition}

In other words, Malcev presentations for $S$ are precisely group presentations for its universal group
involving no inverses of generators.

\begin{proposition}[{\cite[Corollary~1.3.2]{c_phdthesis}}]
\label{prop:finsgmpresifffinugpres}
If a group-embeddable semigroup $S$ has a finite Malcev presentation,
then its universal group $G$ is finitely presented. Conversely, if the
universal group of $S$ is finitely presented and $S$ itself is
finitely generated, then $S$ admits a finite Malcev presentation.
\end{proposition}

Our strategy in proving Theorem \ref{thmmalcev1} relies on a dichotomy:
either $S$ and $T$ are both groups, in which
case the problem reduces to the finite presentability of groups, or else
$S$ and $T$ have isomorphic universal groups. The key technical observation is the following:

\begin{lemma}
\label{lem:rlquotients}
Let $G$ be a group, let $S$ be a subsemigroup of $G$, and let $T$ be a subsemigroup of finite Green index in $S$.
 Then either $T$
is a group or for any $s \in S\setminus T$ there exist $u_s,v_s,w_s,x_s \in T$
with $s = u_sv_s^{-1}$ and $s = w_s^{-1}x_s$ in $G$.
\end{lemma}

\begin{proof}
Let $J$ be the group of units of $T$, if $T$ is a monoid, and otherwise set $J = \varnothing$. If $J = T$ there is nothing to
prove; so suppose $T \neq J$. 
Let $s \in S\setminus T$. Pick any $t \in T \setminus J$ and consider the elements $s, st, st^2,\ldots$. 
Since $T$ has finite Green index in $S$, either we have $st^i\in T$ for some $i$, or else
$st^i\gr st^j$ for some $i<j$.
If $st^i\in T$ the elements $u_s=st^i$ and $v_s=t^i$ belong to $T$ and satisfy $u_sv_s^{-1}=s$.
On the other hand, if $st^i\gr st^j$ then there exists $u\in S^1$ such that $st^ju=st^i$, which implies
$t^{j-i}u=1$, and contradicts the assumption $t\not\in J$.
Similar reasoning
using $\rel{L}$ yields $w_s$ and $x_s$.
\end{proof}

Any finite cancellative semigroup is a group, so for the class of cancellative semigroups the property of being a group is a 
finiteness condition. The following result shows that for cancellative semigroups this property is preserved when taking finite Green index subsemigroups or extensions. 

\begin{proposition}
\label{lem_cancellative}
Let $S$ be a cancellative semigroup and let $T$ be a subsemigroup with finite Green index in $S$. 
Then $S$ is a group if and only if $T$ is a group. 
\end{proposition}
\begin{proof}
In \cite[Theorem~5.1 \& Proposition~5.3]{gray_green1}
it is shown that $T$ is a group if $S$ is a group.

For the converse, suppose that $T$ is a group, say with identity element $e$. Since $S$ is cancellative and $e$ is an idempotent, $e$ is a two-sided identity in $S$. Therefore $S$ is a monoid and $T$ is a subgroup of the group of units of $S$. 
Let $s \in S$ be arbitrary. We claim that $s^i \in T$ for some $i \in \mathbb{N}$. Otherwise, since the Green index is finite there would exist $i < j$ with $s^i \gr^T s^j$, implying $s^j = s^i t$ for some $t \in T$ which by left cancellativity yields $s^{j-i} = t \in T$, a contradiction. Therefore $s^i$ belongs to the group of units of $S$ for some $i \in \mathbb{N}$ which is clearly only possible if $s$ itself is invertible. Thus every element is invertible and we conclude that $S$ is a group. 
\end{proof}

\begin{corollary}
\label{corol:sgrpifftgrp}
Let $G$ be a group, let $S$ be a subsemigroup of $G$, and let $T$ be a subsemigroup of finite Green index in $S$.
Then $T$ is a
group if and only if $S$ is a group.
\end{corollary}

\begin{theorem}
\label{thm:gorisoug}
Let $S$ be a group-embeddable semigroup, and let $T$ be a subsemigroup of finite
Green index. Then either $S$ and $T$ are both groups or $S$ and
$T$ have isomorphic universal groups.
\end{theorem}

\begin{proof}
Let $G$ be the universal group of $S$ and view $S$ and $T$ as being
subsemigroups of $G$. By Corollary~\ref{corol:sgrpifftgrp} either
both $S$ and $T$ are groups or neither are groups. 
In the former case,
the proof is complete.
In the latter case, Lemma~\ref{lem:rlquotients} says that every element of $S\setminus T$ can be
expressed as a right or left quotient of elements of $T$. The proof of
\cite[Theorem~3.1]{crr_finind} thus applies to show that the universal
group of $T$ is isomorphic to $G$.
\end{proof}

The following is now immediate:

\begin{corollary}
\label{corol:ugfi}
Let $S$ be a group-embeddable semigroup, and let $T$ be a subsemigroup of finite
Green index. Let $G$ and $H$ be the universal groups of $S$ and
$T$ respectively. Then $G$ contains a subgroup of finite index isomorphic to $H$.
\end{corollary}

We are now in a position to prove our main result of this section.

\begin{proof}[of Theorem \ref{thmmalcev1}]
Let $G$ and $H$ be the universal groups of $S$ and $T$,
respectively. By Corollary~\ref{corol:ugfi}, $H$ is a finite index
subgroup of $G$; hence, by the Reidemeister--Schreier Theorem
\cite[\S II.4]{lyndon_cgt}, $G$ is finitely presented if and only
if $H$ is finitely presented. Furthermore, from Theorem~\ref{thm_finitegeneration} above $S$ is finitely generated if and only if $T$ is finitely generated.

Now, by the observations in the foregoing paragraph and by using Proposition~\ref{prop:finsgmpresifffinugpres} twice, one sees that:\\
$\phantom{\iff} S$ has a finite Malcev presentation\\
$\iff S$ is finitely generated and $G$ is finitely presented \\
$\iff S$ is finitely generated and $H$ is finitely presented \\
$\iff T$ is finitely generated and $H$ is finitely presented \\
$\iff T$ has a finite Malcev presentation.
\end{proof}

\begin{remark}
It is natural to ask whether preservation of finite presentability when passing to subsemigroups of finite Green index holds for other types of presentations, e.g. presentations of cancellative semigroups, left or right cancellative semigroups, or inverse semigroups.
The corresponding results for finite Rees index are known
(\cite[Theorems 2, 3]{crr_finind} and \cite[Theorem 1.2]{inverserees}),
but rely on the result for the `ordinary' presentations \cite[Theorem 1.3]{Ruskuc1}.
Consequently, for Green index, these results either have to wait for a positive solution to Problem \ref{presentationsquestion}, or else entirely new methods are required.
\end{remark}

The method of proof used above reduces either to the case where $S$
and $T$ are both groups, or to the case where, as for finite Rees
index, every element of $S$ can be expressed as a right or left
quotient of $T$. In light of this, one might suspect that perhaps
finite Green index for group-embeddable semigroups reduces either to
finite group index or to finite Rees index. The following example
dispels these suspicions:

\begin{example}
Let $n \in \nset$. Let $S = \zset \times (\nset \cup\{0\})$ and let $T
= \zset \times ((\nset \cup\{0\}) - \{1,\ldots,n\})$. Then $S$ and $T$
are both group-embeddable and $T$ is a subsemigroup of $S$. Furthermore,
\[
S-T = \zset \times \{1,\ldots,n\}.
\]
Let $k \in \{1,\ldots,n\}$. Then for any $z \in \zset$, the
$\rel{R}^T$-class of $(z,k)$ is $\zset \times \{k\}$. Since $S$ is
commutative, these are the $\rel{L}^T$ and thus the
$\rel{H}^T$-classes. Therefore there are only $n$ different
$\rel{H}^T$-classes in $S-T$. Thus $T$ has finite Green index in
$S$. Since $S-T$ is infinite, $T$ does not have finite Rees index in
$S$. Furthermore, neither $S$ nor $T$ are groups.
\end{example}

\section{The Word Problem}
\label{secwp}

In this section we consider some questions relating to decidability. Recall that for a semigroup $S$ finitely generated by a set $A$ we say that $S$ has a \emph{soluble word problem} (with respect to $A$) if there exists an algorithm which, for any two words $u,v \in A^+$, decides whether the relation $u=v$ holds in $S$ or not. For finitely generated semigroups it is easy to see that solubility of the word problem does not depend on the choice of (finite) generating set for $S$.

The following result concerning the word problem essentially follows from the arguments in the proof of Theorem~\ref{thm_pres_smalltobig}.

\begin{theorem}\label{thm_soluble}
Let $S$ be a finitely generated semigroup with $T$ a subsemigroup of $S$ with finite Green index. Then $S$ has soluble word problem  if and only if $T$ and all the relative Sch\"{u}tzenberger groups of $S \setminus T$ have soluble word problem.
\end{theorem}
\begin{proof}
\begin{sloppypar}
Assume that $T$ has soluble word problem and that all of the relative Sch\"{u}tzenberger groups $\Gamma_i$ of $S \setminus T$ have soluble word problem. By Theorem~\ref{thm_main_fg}, $T$ is generated by a finite subset $B \subseteq T$ say, and $S = \lb A \rb$ where $A = B \cup \{ h_i : i \in I \}$. Theorem~\ref{thm_pres_smalltobig} gives a (possibly infinite) presentation for $S$ but where the sets of relations \eqref{22} and \eqref{33} are both finite since $A$, $B$ and $I$ are all finite.
\end{sloppypar}

Let $w_1, w_2 \in A^+$. As in the proof of Theorem~\ref{thm_pres_smalltobig} using the relations  \eqref{22} and \eqref{33} we can rewrite $w_1$ into a word of the form $w_1'' d_j$  where $w_1'' \in B^+$ and similarly rewrite $w_2$ into a word of the form $w_2'' d_k$ with  $w_2'' \in B^+$. By Lemma~\ref{lemma_bookkeeping}(i) $w_1$ represents an element of $T$ if and only if $j=1$, while $w_2$ represents an element of $T$ if and only if $k=1$. So if $j=1$ and $k \neq 1$ (or vice versa) we deduce that $w_1 \neq w_2$. If $j=k=1$ then $w_i = w_i'' \in B^+$  ($i = 1,2$) and $w_1 = w_2$ if and only if $w_1'' = w_2''$ in $T$ which can be decided since $T$ has soluble word problem.

The remaining possibility is that $j \neq 1$ and $k \neq 1$ so $w_1$ and $w_2$ both represent elements from $S \setminus T$. Now, again following the argument in the proof of Theorem~\ref{thm_pres_smalltobig} using the relations \eqref{22} and \eqref{33} we deduce:
\[
w_1 = d_r \overline{\xi} (\overline{w_1}),\ 
w_2 = d_r \overline{\xi} (\overline{w_2})
\]
where $\overline{w_1}, \overline{w_2} \in C_k^+$. Now $w_1 = w_2$ in $S$ if and only if $\overline{w_1} = \overline{w_2}$ in the Sch\"{u}tzenberger group $\Gamma_k$ and this can be decided since $\Gamma_k$ has soluble word problem by assumption.

For the converse, suppose that $S$ has soluble word problem. Then immediately $T$ has soluble word problem since it is a finitely generated subsemigroup of $S$. Finally let $H$ be a $T$-relative $\gh$-class in $S \setminus T$, with fixed representative $h \in H$. The group $\Gamma = \Gamma(H) = \mathrm{Stab}(H) / \gamma$ is finitely generated by Theorem~\ref{thm_schutz_gen}. Let $Y$ be a finite subset of $\mathrm{Stab}(H)$ such that $\lb Y / \gamma \rb = \Gamma(H)$. Let $w_1, w_2 \in     (Y / \gamma)^*$ Then $w_i = w_i' / \gamma$ where $w_i' \in B^*$ ($i = 1,2$) and $w_1 = w_2$ if and only if $hw_1' = hw_2'$ in $S$ which is decidable since $S$ is assumed to have soluble word problem.
\end{proof}

As with other results in this article,
Theorem~\ref{thm_soluble} generalises the well-known classical result from group theory and the corresponding result for finite Rees index subsemigroups proved in \cite{Ruskuc1}.
Just as for
Theorems~\ref{thm_aptrick1} and \ref{thm_aptrick2},
Theorem~\ref{thm_soluble} may be used to prove that a finitely generated semigroup with finitely many left and right ideals has soluble word problem if and only if all of its Sch\"{u}tzenberger groups have soluble word problem.

A finitely generated group $G$ has only finitely many subgroups of any given finite index $n$.
If $G$ is finitely presented, then a list of generating sets of all these subgroups can be obtained effectively.
In \cite[Corollary~32]{gray_green1} it was shown that the first of these two facts generalises to semigroups: a finitely generated semigroup has only finitely many subsemigroups of any given finite Green index $n$. We now show that the second statement does not generalise to semigroups and Green index.

\begin{theorem}
There  does not exist an algorithm which would take as its input a finite semigroup presentation (defining a semigroup $S$)
and a natural number $n$, and which would return as the output a list of generators of all subsemigroups of $S$ of Green index $n$.
\end{theorem}
\begin{proof}
Let $S_0$ denote $S$ with a zero element $0$ adjoined. The Green index of the subsemigroup $\{ 0 \}$ in $S_0$ is equal to $| S_0 \setminus \{0 \} | = |S|$.
This observation along with the argument \cite[Theorem~5.5]{Ruskuc&Thomas} suffices to prove the theorem.
\end{proof}

\section{Growth}
\label{secgrowth}

A (discrete) \emph{growth function} is a monotone non-decreasing function from $\mathbb{N}$ to $\mathbb{N}$. For growth functions $\alpha_1, \alpha_2$ we write $\alpha_1 \preccurlyeq \alpha_2$ if there exist natural numbers $k_1, k_2 \geq 1$ such that
$\alpha_1(t) \leq k_1 \alpha_2(k_2 t)$
for all $t \in \mathbb{N}$. We define an equivalence relation on growth functions
by $\alpha_1 \sim \alpha_2$ if and only if $\alpha_1 \preccurlyeq \alpha_2$ and $\alpha_2 \preccurlyeq \alpha_1$.
The $\sim$-class $[\alpha]$ of a growth function $\alpha$ is called the
\emph{growth type} or just \textit{growth} of the function $\alpha$.

Let $S$ be a semigroup and let $X$ be a subset of $S$. Note that we do not insist here that $X$ generates $S$. Then for $s \in S^1$ and $n \in \mathbb{N}$ we define:
\[
\outball_X(s,n) = \{ s x_1 \ldots x_r \in S : x_i \in X^1, r \leq n \}
\]
and call this the \emph{out-ball of radius $n$ around $s$ with respect to $X$}.
For a
semigroup $S$ generated by a finite set $A$ the function
\[
g_S: \mathbb{N} \rightarrow \mathbb{N},
\quad
g_S(m) = |\outball_A(1,m)|
\]
is called the \emph{growth function} of the semigroup $S$.
It is well-known (and easily proved) that the growth type of a semigroup is independent of the choice of finite generating set. Also note that if $T$ is a finitely generated subsemigroup of a finitely generated semigroup $S$ then $g_T \preccurlyeq g_S$ (since we may take a finite generating set for $S$ that contains a finite generating set for $T$). In general the converse is not true, but it is in the case that $S$ is a group and $T$ is a subgroup of finite index (this follows from the more general fact that growth type is a quasi-isometry invariant; see \cite[p115, Section 50]{delaHarpe2000}).
Here we shall show that this fact is more generally true for subsemigroups of finite Green index. In fact, the result goes through under far weaker hypotheses as we now see.
The following result is very straightforward to prove and it is quite likely that it is already known. We include it here for completeness.

\begin{proposition}
Let $S$ be a semigroup and let $T$ be a subsemigroup of $S$.
Suppose that $T$ is finitely generated and that there exists a finite subset $R$ of $S^1$ with $1 \in R$ and $S^1 = RT^1$. Then $S$ and $T$ are both finitely generated and have the same type of growth.
\end{proposition}
\begin{proof}
Let $B \subseteq T$ be a finite generating set for $T$ and define $A = B \cup R$ which is clearly a finite generating set for $S$.
For $t \in T$ let $l_B(t)$ be the shortest length of a word in $B^+$ representing $t$ (i.e. the length of the element $t$ with respect to $B$).
Now $g_T \preccurlyeq g_S$ since $T \leq S$ so we just have to prove $g_S \preccurlyeq g_T$.

As in Lemma~\ref{lem_sigmaandtau}, for all $a_1, a_2 \in A$ there exists $r = r(a_1, a_2) \in R$ and $\mu(a_1, a_2) \in T^1$ satisfying:
\begin{equation}
a_1 a_2 = r(a_1, a_2) \mu(a_1, a_2). \label{mix}
\end{equation}
We claim that with $k_1 = |R|$ and $k_2 = \max\{ l_B(\mu(a_1,a_2)) :a_1, a_2 \in A  \}$ we have
\[
g_S(n) \leq k_1 g_T(k_2n)
\]
for all $n \in \mathbb{N}$. Indeed, applying \eqref{mix}, given any word $a_1 \ldots a_k \in A^+$ there exists $r \in R$ and $\mu_i \in \{ \mu(a_1,a_2)) :a_1, a_2 \in A \}$ ($i \in \{1,\ldots,k\}$) with:
\[
a_1 \ldots a_k = r \mu_1 \ldots \mu_k.
\]
(This is proved in much the same way as the first part of Lemma~\ref{lemma_bookkeeping}.)
For all $i = 1, \ldots, k$ we have $\mu_i \in B^+$ and $l_B(\mu_i) \leq k_2$.
It follows that for all $n \in \mathbb{N}$:
\begin{equation}
\outball_A (1,n)
\subseteq
\bigcup_{r \in R}
\outball_B (r,k_2 n).
\label{balls_eq}
\end{equation} 
But for all $s \in S$ and $m \in \mathbb{N}$ clearly we have:
\[
| \outball_B (s,m) | \leq | \outball_B (1,m ) |.
\]
Therefore by \eqref{balls_eq}:
\[
g_S(n) =
|\outball_A(a,n) \leq
|R||\outball_B(1,k_2n)| = 
k_1 g_T (k_2 n),
\]
for all $n \in \mathbb{N}$. 
\end{proof}

\begin{corollary}
Let $S$ be a semigroup and let $T$ be a subsemigroup of finite Green index. 
Then $S$ is finitely generated if and only if $T$ is finitely generated, in which case $S$ and $T$ have the same type of growth. 
\end{corollary}

\section{Automaticity}
\label{secautomatic}

In this section we apply our results concerning generators and rewriting
to investigate how the property of being automatic behaves with respect to finite Green index subsemigroups.
In what follows we will give a very rapid summary of the basic definitions; for a better paced introduction we
refer the reader to
\cite{campbell_autsg}, \cite{hoffmann_relatives}, or
\cite{c_phdthesis}.

Following \cite{epstein_wordproc}, and unlike the previous sections, throughout this section
we will make a strict distinction between a word over an alphabet and the element of the semigroup this word represents. Let $A$ be an alphabet representing a generating set for a
semigroup $S$. If $w$ is a word in $A^+$, it represents an element
$\elt{w}$ in $S$. If $K \subseteq A^+$, then $\elt{K}$ denotes the set
of elements of $S$ represented by at least one word in $K$.

Now suppose $A$ and $B$ are two alphabets, and let $\$$ be a symbol belonging to neither.
Let
$C = \{(a,b) : a,b \in A \cup \{\$\}\} - \{(\$,\$)\}$ be a new
alphabet. Define the mapping $\delta : A^+ \times A^+ \to C^+$
by
\[
(u_1\cdots u_m,v_1\cdots v_n) \mapsto
\begin{cases}
(u_1,v_1)\cdots(u_m,v_n) & \text{if }m=n,\\
(u_1,v_1)\cdots(u_n,v_n)(u_{n+1},\$)\cdots(u_m,\$) & \text{if }m>n,\\
(u_1,v_1)\cdots(u_m,v_m)(\$,v_{m+1})\cdots(\$,v_n) & \text{if }m<n,
\end{cases}
\]
where $u_i\in A$, $v_i \in B$.

Suppose now that $L$ is a regular language over $A$ such that
$\elt{L} = S$. For any $w \in A^*$, define the relation
\[
L_w = \{(u,v) : u,v \in L, \elt{uw} = \elt{v}\}.
\]

The pair $(A,L)$ forms an \defterm{automatic structure} for $S$ if the
language $L_a\delta$ is regular for each $a \in A \cup
\{\emptyword\}$. An \defterm{automatic semigroup} is a semigroup that
admits an automatic structure.

Our main result for this section is:

\begin{theorem}
\label{thm:fgiautdown}
Let $S$ be a semigroup and let $T$ be a subsemigroup of $S$ of finite Green index. If $S$ is automatic, then $T$ is automatic.
\end{theorem}

\begin{proof}
Suppose that $S$ admits an automatic structure  $(A,L)$.
All the notation fixed in Section \ref{sec_rewriting} will remain in force throughout this proof.
The goal is
to construct an automatic structure for $T$. 
The proof is based on the rewriting
technique given in Lemma~\ref{lemma_bookkeeping} and Theorem~\ref{thm_finitegeneration} above.

In Theorem ~\ref{thm_finitegeneration} we proved that the set
\begin{equation}
\label{eqaut1}
\{  
\tau(i, \sigma(\elt{a},j)) : i,j \in I^1, \; a \in A
\}
\end{equation}
generates $T$.
More precisely, we proved that an element $\elt{a_1}\, \elt{a_2}\ldots \elt{a_n}\in T$, where $a_i\in A$,
can be re-written as
$$
\elt{a_1}\, \elt{a_2}\ldots \elt{a_n} = 
\tau(j_1, \sigma(\elt{a_1}, i_1)) \tau(j_2, \sigma(\elt{a_2}, i_2)) \ldots \tau(j_n, \sigma(\elt{a_n}, i_n)),
$$
where the indices $i_k$, $j_k$ are computed by the following recursion:
\begin{enumerate}[(i)] 
\item $i_n = 1$,
\item $i_{k-1} = \rho(\elt{a_k},i_k)$ for $k =n,n-1,\ldots,2$,
\item $j_1 = \rho(\elt{a_1},i_1)$,
\item $j_{l+1} = \lambda(j_l,\sigma(\elt{a_l},j_l))$ for $l=1,2,\ldots,n-1$,
\item $\lambda(j_n,\sigma(\elt{a_n},i_n)) = 1$.
\end{enumerate}

Let us introduce a new alphabet representing the elements of (\ref{eqaut1}):
$$
B=\{b_{j,a,i}\::\: i,j\in I^1,\ a\in A\},\ \elt{b_{j,a,i}}=\tau(j, \sigma(\elt{a},i)).
$$
Let $R \subseteq A^+ \times B^+$ be the relation consisting of pairs of strings
\begin{equation}
\label{eq:transrel}
(a_1a_2\cdots a_n,b_{j_1,a_1,i_1}b_{j_2,a_2,i_2}\cdots b_{j_n,a_n,i_n})
\end{equation}
such that the properties (i)--(v) above are satisfied.
Notice in particular the correspondence between the letters of $a_i$
and the middle subscripts of the letters $b_{j_k,a_k,i_k}$ in
\eqref{eq:transrel}. It is clear that the set of pairs of \emph{all} strings
\eqref{eq:transrel} -- or rather the image of this under $\delta$ --
is regular. An automaton recognizing this set can easily be adapted to
check the properties (i)--(v): conditions~(i), (iii), and~(v) are all single
`local' checks, and conditions~(ii) and~(iv) require only that the automaton
store the subscripts from the previously read letter of $B$. Thus the
language $R\delta$ is regular.

We now have:
\begin{enumerate}[(i)]
\setcounter{enumi}{5}
\item If $u \in A^+$ represents an element of $T$, then there is a unique string $v \in B^+$ with $(u,v) \in R$ and $\elt{u} = \elt{v}$.
\item If $v = b_{j_1,a_1,i_1}b_{j_2,a_2,i_2}\cdots b_{j_n,a_n,i_n} \in B^+$ satisfies conditions (i)--(v), then there is a unique string $u \in A^+$ with $(u,v) \in R$.
\item If $(u,v) \in R$ then $\elt{u} = \elt{v}$ and so $\elt{u} \in T$.
\end{enumerate}

Let $M = \{v \in B^+ : (\exists u \in L)((u,v) \in R)\}$. The aim is
now to show that $(B,M)$ is an automatic structure for $T$. Clearly, the language $M$ maps onto $T$.

Let $b \in B$. Let $w \in A^+$ be such that $\elt{w} = \elt{b}$. The
language $L_w\delta$ is regular by
\cite[Proposition~3.2]{campbell_autsg}.
The language $(R^{-1} \circ L_w \circ R)\delta$ is thus also regular
and
\begin{align*}
&\;\phantom{\iff}\;(u,v) \in R^{-1} \circ L_w \circ R\\
&\iff u,v \in M \land (\exists p,q \in L)((u,p) \in R^{-1} \land (p,q) \in L_w \land (q,v) \in R)\\
&\iff u,v \in M \land (\exists p,q \in L)(\elt{p} = \elt{u} \land \elt{q} = \elt{v} \land \elt{pw} = \elt{q})\\
&\iff u,v \in M \land (\exists p,q \in L)(\elt{p} = \elt{u}  \land \elt{q} = \elt{v} \land \elt{u}\,\elt{w} = \elt{v})\\
&\iff u,v \in M \land \elt{u}\,\elt{w} = \elt{v}\\
&\iff u,v \in M \land \elt{u}\,\elt{b} = \elt{v}\\
&\iff (u,v) \in M_b.
\end{align*}
Thus $M_b = R^{-1} \circ L \circ R$. So $M_b\delta$ is regular
and so $(B,M)$ is an
automatic structure
for $T$.
\end{proof}

This theorem provides a common generalisation of the corresponding group theoretic result
 \cite[Theorem~4.1.4]{epstein_wordproc}
(without relying on the geometric `fellow traveller' property) and
\cite[Theorem~1.1]{hoffmann_autofinrees} for Rees index.

A variation of the notion of automatic semigroup is that of an asynchronously automatic semigroup.
Here we require that each relation $L_a$ (for $a\in A\cup\{\epsilon\})$ is recognised by an asynchronous two-tape automaton;
see \cite[Definition 3.3]{hoffmann_relatives} for details.
The proof of Theorem \ref{thm:fgiautdown} carries over verbatim to the asynchrnous case; the reference to
\cite[Proposition~3.2]{campbell_autsg} should be replaced by \cite[Proposition~2.1(3)]{hoffmann_relatives}.
Thus we have:

\begin{theorem}
\label{thm:fgiautdown1}
Let $S$ be a semigroup and let $T$ be a subsemigroup of $S$ of finite Green index. If $S$ is asynchronously automatic, then $T$ is asynchronously automatic.
\end{theorem}

The converses of Theorems~\ref{thm:fgiautdown} and \ref{thm:fgiautdown1} do not hold.
We demonstrate this by using the following example, which was introduced
in \cite[Example~5.1]{campbell_autcompsimple} for a different
purpose, viz., to show that a Clifford semigroup whose group
maximal subgroups are all automatic need not itself be automatic:

\begin{example}
Let $F$ be the free group on two generators $a,b$, and let $G$ be the free product
of two cyclic groups of order 2, i.e. $G=\langle c,d \:|\: c^2=d^2=1\rangle$.
Let $\phi:F\rightarrow G$ be the epimorphism defined by $a\mapsto c$, $b\mapsto d$.
Form the strong semilattice $S=\mathcal{S}(F,G,\phi)$.

Now, $F$, being a finitely
generated free group, is automatic. 
Furthermore, $F$ has finite Green index in $S$, with $G$ a unique $\gh$-class in $S\setminus F$. The Sch\"{u}tzenberger group of this $\gh$-class is $G$, and so is automatic.
But in \cite[Example~5.1]{campbell_autcompsimple} it is proved that $S$ is
not automatic. 
We will actually go further and show $S$ is not even
asynchronously automatic. 

Suppose for \textit{reductio ad absurdum} that $(A,L)$ is an
asynchronous automatic structure for $S$. Let $A_F = \{a \in A :
\elt{a} \in F\}$. Let $L_F = L \cap (A_F)^+$ and $L_G = L \setminus
L_F$. 
Then $\elt{L_G} = G$. Choose a representative $w \in L_G$ of
the identity $1_G$ of $G$. Construct the rational relation
$L_w$. Let $K = \{u : (u,w) \in L_w\}$; then $K$ is regular and
represents all those elements $s$ of $S$ with $s1_G = 1_G$. Therefore,
by the definition of $S$ we have
$\elt{K} = \{1_G\} \cup ( 1_G\phi^{-1})$.
Let $J = \{u : (u,w) \in L_\emptyword\}$. Then $J$ is regular and
consists of all elements of $L$ that represent $1_G$. Thus $K \setminus J$ is
regular and represents the kernel (in the group-theoretical sense) of
$\phi$. Thus this kernel, $\elt{K - J}$, is a rational subset of
the free group $F$. But it is thus a non-trivial normal rational
subgroup of infinite index in $F$, which is a contradiction by
\cite[Corollary~4]{frougny89} and \cite[Theorem~1]{karass57}.
\end{example}

\begin{remark}
There are other possible definitions of automaticity depending on which side generators act, and on which side the padding symbols are placed;
see \cite{hoffmann03}.
Straightforward modification of the above argument yields the corresponding results for each of them.
\end{remark}

\section*{Acknowledgements}

We thank Simon Craik and Victor Maltcev for useful discussions during the preparation of this paper. 

\bibliographystyle{abbrv}

\end{document}